\documentclass[11pt]{article}
\usepackage[dvips]{graphicx}
\usepackage[latin1]{inputenc}
\usepackage{color}
\usepackage{amsmath,amssymb,amsthm,mathrsfs,amsfonts}

\newtheorem{theorem}{Theorem}[section]
\newtheorem{definition}[theorem]{Definition}
\newtheorem{lemma}[theorem]{Lemma}
\newtheorem{proposition}[theorem]{Proposition}

\title{On line colorings of finite projective spaces\thanks{Research supported by: G. A-P. partially supported by CONACyT-M{\' e}xico under Projects 282280 and PAPIIT-M{\' e}xico under Project IN107218. 
Gy. K. partially supported by the bilateral Slovenian-Hungarian Joint Research Project, grant no. NN 114614 (in Hungary) and N1-0032 (in Slovenia).
C. R-M. partially supported by PAPIIT-M{\' e}xico under Project IN107218. A. V-{\' A}. partially supported by SNI of CONACyT-M{\' e}xico.}}

\author{Gabriela Araujo-Pardo \footnotemark[2] \and
Gy{\" o}rgy Kiss \footnotemark[3] \and
Christian Rubio-Montiel \footnotemark[5] \and
Adri{\' a}n V{\' a}zquez-{\' A}vila \footnotemark[4]}

\begin{document}

\maketitle

\def\thefootnote{\fnsymbol{footnote}}
\footnotetext[2]{Instituto de Matem{\' a}ticas, Universidad Nacional Aut{\' o}noma de M{\' e}xico, Ciudad Universitaria, 04510, Mexico City, Mexico, {\tt garaujo@matem.unam.mx}.}
\footnotetext[5]{Divisi{\' o}n de Matem{\' a}ticas e Ingenier{\' i}a at FES Acatl{\' a}n, Universidad Nacional Aut{\' o}noma de M{\' e}xico, 53150, State of Mexico, Mexico, {\tt christian.rubio@acatlan.unam.mx}.}
\footnotetext[3]{Department of Geometry and MTA-ELTE Geometric and Algebraic Combinatorics Research Group, E{\" o}tv{\" o}s Lor{\' a}nd University, H-1117 Budapest, P{\' a}zm{\' a}ny s. 1/c, Hungary, and FAMNIT, University of Primorska, 6000 Koper, Glagolja\v ska 8, Slovenia, {\tt kissgy@cs.elte.hu}.}
\footnotetext[4]{Subdirecci{\' o}n de Ingenier{\' i}a y Posgrado, Universidad Aeron{\' a}utica en Quer{\' e}taro, 
Parque Aeroespacial Quer{\' e}taro, 76270, Quer{\' e}taro, M{\' e}xico, {\tt adrian.vazquez@unaq.edu.mx}.}

\begin{abstract}
In this paper, we prove lower and upper bounds on the achromatic and the pseudoachromatic indices 
of the $n$-dimensional finite projective space of order $q$.
\end{abstract}




\section{Introduction}

The results given in this paper are related to the well-known combinatorial problem
called the Erd{\H o}s-Faber-Lov{\' a}sz Conjecture (for short EFL Conjecture), see  \cite{MR0409246}.

Let $\mathbf{S}$ be a finite linear space. A \emph{coloring} of $\mathbf{S}$ with $k$ colors is an assignment of the lines of $\mathbf{S}$ to a set of colors $[k]:=\{1,\dots,k\}$. A coloring of $\mathbf{S}$ is called \emph{proper} if any two intersecting lines have different colors. The \emph{chromatic index} $\chi'(\mathbf{S})$ of $\mathbf{S}$ is the smallest $k$ such that there exists a proper coloring of $\mathbf{S}$ with $k$ colors. Erd{\H o}s, Faber and Lov{\' a}sz conjectured (\cite{MR0409246,MR602413}) that the chromatic index of any finite linear space $\mathbf{S}$ cannot exceed 
the number of its points, so if $\mathbf{S}$ has $v$ points then\[\chi'(\mathbf{S})\leq v.\]

In \cite{MR1038405} the EFL Conjecture was proved for one of the most studied linear spaces, namely for 
the $n$-dimensional finite projective space of order $q$, $\mathrm{PG}(n,q).$ In this case it is known that 
\[\chi'(\mathrm{PG}(n,q))\leq \frac{q^{n+1}-1}{q-1}.\]

Three of this article's authors proved the EFL Conjecture for some linear spaces (\cite{ARV,MR3560871}). Moreover, in \cite{MR3249588,MR3774452,MR2778722} two of them have considered different types of colorations that expand the notion of the chromatic index for graphs: the achromatic and the pseudoachromatic indices. 
Related problems were intensively studied by several authors, see \cite{MR543176,MR0256930,MR0272662,MR989126}. Furthermore, in \cite{MR695809} Colbourn and Colbourn investigated these parameters for block designs (see also \cite{MR1178507}).

A coloring of $\mathbf{S}$ is called \emph{complete} if each pair of colors appears on at least one point of $\mathbf{S}$. It is not hard to see that any proper coloring of $\mathbf{S}$ with $\chi'(\mathbf{S})$ colors is a complete coloring. The \emph{achromatic index} $\alpha'(\mathbf{S})$ of $\mathbf{S}$ is the largest $k$ such that there exists a proper and complete coloring of $\mathbf{S}$ with $k$ colors. The \emph{pseudoachromatic index} $\psi'(\mathbf{S})$ of $\mathbf{S}$ is the largest $k$ such that there exists a complete coloring (not necessarily proper) of $\mathbf{S}$ with $k$ colors. Clearly we have that 
\begin{equation}\label{iquality:chi_alpha_psi}
\chi'(\mathbf{S}) \leq \alpha'(\mathbf{S}) \leq \psi'(\mathbf{S}).
\end{equation}

If $\Pi _q$ is an arbitrary (not necessarily desarguesian) finite projective plane of order $q,$ then 
$$\chi'(\Pi _q) =  \alpha'(\Pi _q) = \psi'(\Pi _q)=q^2+q+1,$$
because any two lines of $\Pi _q$ have a point in common. The situation is much more complicated in higher dimensional
projective spaces, the exact values of the chromatic indices are not known for $n\geq 3.$ 
The aim of this paper is to study the achromatic and pseudoachromatic indices of finite projective spaces. 
Our main results are summarized in the following theorem.



\begin{theorem}\label{thm:uno}
Let $v=\frac{q^{n+1}-1}{q-1}$ denote the number of points of $\mathrm{PG}(n,q)$. 
\begin{enumerate}
\item
If $n=3\cdot 2^{i}-1$ $(i=1,2,\dots )$ then 
\begin{equation}
\label{main1}
c_n\frac{1}{q}v^{\frac{4}{3}+\frac{1}{3n}} < \alpha'(\mathrm{PG}(n,q)),
\end{equation} 
where $\frac{1}{2^{\frac{7}{5}}}\leq c_n< \frac{1}{2^{\frac{4}{3}}}$ is a 
constant that depends only on $n$.
\item
If $n\geq 2$ is an arbitrary integer then $$\psi'(\mathrm{PG}(n,q))<\frac{1}{q}v^{\frac{3}{2}}.$$
\end{enumerate}
\end{theorem}

In Section \ref{ProjectiveSpaces}, we collect some known properties of projective spaces, spreads and packings, 
and we prove a lemma about the existence of a particular spread. 
In Section \ref{section3}, we prove the main theorems about the achromatic and pseudoachromatic indices. 
Finally, in Section \ref{new} (an Appendix is attached, where) we consider the smallest projective space, $\mathrm{PG}(3,2),$ and determine the exact value of its pseudoachromatic index without using a computer.


\section{On projective spaces}\label{ProjectiveSpaces}

It is well-known that, for any $n>2$, the $n$-dimensional finite projective space of order $q$ exists if and only if $q$ is a prime power and it is unique up to isomorphism. Let $V_{n+1}$ be an $(n+1)$-dimensional vector space over the Galois field $\mathrm{GF}(q)$ with $q$ elements. The \emph{$n$-dimensional finite projective space}, denoted by $\mathrm{PG}(n,q),$ is the geometry whose $k$-dimensional subspaces for $k=0,1,\ldots ,n$ are the $(k+1)$-dimensional subspaces of  $V_{n+1}.$ For the detailed description of these spaces we refer the reader to \cite{MR1612570}.

The basic combinatorial properties of $\mathrm{PG}(n,q)$ can be described 
by the $q$-nomial coefficients $\genfrac[]{0pt}{2}{n}{k} _q.$ This number is defined as
$$ \genfrac[]{0pt}{0}{n}{k} _q := 
\frac{(q^n-1)(q^n-q)\ldots (q^n-q^{k-1})}{(q^k-1)(q^k-q)\ldots (q^k-q^{k-1})},$$
and it equals to the
number of $k$-dimensional subspaces in an $n$-dimensional vector space over 
$\mathrm{GF}(q).$ 
The proof of the following proposition is straightforward. 
\begin{proposition}
\label{combprop}
The following holds in $\mathrm{PG}(n,q)$:
\begin{itemize}
\item
the number of $k$-dimensional subspaces is 
$\genfrac[]{0pt}{1}{n+1}{k+1} _q,$
in particular, the number 
of points equals to $\frac{q^{n+1}-1}{q-1}$
and the number of lines equals to 
$\frac{(q^{n+1}-1)(q^{n}-1)}{(q^{2}-1)(q-1)};$
\item
the number of $k$-dimensional subspaces through a given 
$m$-dimensional $(m\leq k)$ subspace is 
$\genfrac[]{0pt}{1}{n-m}{k-m} _q.$

\end{itemize}
\end{proposition}

A \emph{$t$-spread} $\mathscr{S}^t$ of $\mathrm{PG}(n,q)$ is a set of $t$-dimensional subspaces (for short $t$-subspaces) of $\mathrm{PG}(n,q)$ that partitions $\mathrm{PG}(n,q)$. That is, each point of $\mathrm{PG}(n,q)$ lies in exactly one element of $\mathscr{S}^t$. Hence any two elements of $\mathscr{S}^t$ are disjoint. A $1$-spread is also called \emph{line spread} and it is denoted by $\mathscr{S}$. It is well-known that a \emph{$t$-spread} of $\mathrm{PG}(n,q)$ exists if and only if $(t+1)|(n+1),$ hence line spreads exist in projective spaces of odd dimension.

A \emph{$t$-packing} $\mathscr{P}^t$ of $\mathrm{PG}(n,q)$ is a 
partition of the $t$-spaces of $\mathrm{PG}(n,q)$ into $t$-spreads. 
A $1$-packing is also called \emph{line packing} or \emph{parallelism} and 
it is denoted by $\mathscr{P}.$ The next result is an 
obvious corollary of Proposition \ref{combprop}.

\begin{proposition}
\label{combprop-2}
$ $
\begin{itemize}
\item
A $t$-spread in $\mathrm{PG}(n,q)$ consists of $\frac{q^{n+1}-1}{q^{t+1}-1}$
$t$-subspaces.
\item
A $t$-packing in $\mathrm{PG}(n,q)$ consists of $\genfrac[]{0pt}{1}{n}{t} _q$
$t$-spreads.

\end{itemize}
\end{proposition}

A necessary and sufficient condition for the existence of a
\emph{$t$-packing} of $\mathrm{PG}(n,q)$ is not known in general.
The following theorems give specific constructions in some particular cases.

\begin{theorem}[Beutelspacher \cite{MR0341270}]
\label{Beutelspacher2}
If $n=2^{i}-1$ with $i=1,2,\dots ,$ then for every prime power $q$ the 
finite projective space $\mathrm{PG}(n,q)$ admits a line packing.
\end{theorem}

\begin{theorem}[Baker \cite{MR0416937}]
For all integers $m>0$ 
the finite projective space $\mathrm{PG}(2m+1,2)$ admits a line packing. 
\end{theorem}


A \emph{regulus} of $\mathrm{PG}(3,q)$ is a set $\mathscr{R}$ of $q+1$ mutually skew 
lines such that any line of $\mathrm{PG}(3,q)$ intersecting three distinct elements of 
$\mathscr{R}$ intersects all elements of $\mathscr{R}$. It is known \cite{MR0250182}  
that any three pairwise skew lines $\ell_1,\ell_2,\ell_3$ of $\mathrm{PG}(3,q)$ are 
contained in exactly one regulus $\mathscr{R}=\mathscr{R}(\ell_1,\ell_2,\ell_3)$ of $\mathrm{PG}(3,q)$. 
A line spread $\mathscr{S}$ of $\mathrm{PG}(3,q)$ is called \emph{regular}, if for any 
three distinct lines of $\mathscr{S}$ the whole regulus $\mathscr{R}=\mathscr{R}(l_1,l_2,l_3)$ 
is contained in $\mathscr{S}$.

\begin{theorem}[Beutelspacher \cite{MR0341270}] 
\label{Beutelspacher}
For any regular spread $\mathscr{S}$ of $\mathrm{PG}(3,q)$ there is a 
packing $\mathscr{P}$ of $\mathrm{PG}(3,q)$ which contains $\mathscr{S}$ 
as one of its spreads.
\end{theorem}


There is an important class of spreads. The notion of geometric spread
was introduced by Segre \cite{MR0169117} in the following way. Let $\left\langle X,Y\right\rangle $ be 
the subspace of $\mathrm{PG}(n,q)$ generated by $X$ and $Y$, where $X$ and $Y$ are  
two different elements of a $t$-spread $\mathscr{S}^t$ of 
$\mathrm{PG}(n,q).$ As $X$ and $Y$ are disjoint, from the dimension formula we get that 
$\left\langle X,Y\right\rangle $ is a $(2t+1)$-subspace. We say that 
$\mathscr{S}^t$ \emph{induces a spread} 
$\mathscr{S}^t_{\left\langle X,Y\right\rangle}$ in $\left\langle X,Y\right\rangle $, 
if any element $Z$ of $\mathscr{S}^t$ having at least one point in $\left\langle X,Y\right\rangle $ 
is totally contained in $\left\langle X,Y\right\rangle $. 
The $t$-spread
$\mathscr{S}^t$ is called \emph{geometric} if $\mathscr{S}^t$ induces a 
spread $\mathscr{S}^t_{\left\langle X,Y\right\rangle}$ in $\left\langle X,Y\right\rangle $ 
for any two distinct elements $X$ and $Y$ of $\mathscr{S}^t$. 

It is not difficult to check (see \cite{MR0341270}, Section 4) that a $t$-spread $\mathscr{S}^t$ of $\mathrm{PG}(n,q)$ 
is geometric if and only if the following holds. If the elements $X$ of 
$\mathscr{S}^t$ are called \emph{large points}, and 
for disjoint elements $X,Y$ of $\mathscr{S}^t$ the subspaces 
$\left\langle X,Y\right\rangle $ 
are called \emph{large lines}, then the large points and large lines form a projective space.
This space, $\Pi_{\mathscr{S}^t},$ has dimension $s=\frac{n+1}{t+1}-1$ and order $q^{t+1}$, it is
isomorphic to $\mathrm{PG}\left( \frac{n+1}{t+1}-1,q^{t+1} \right) .$ 

The following two results are due to Segre \cite{MR0169117}.

\begin{theorem} 
\label{segre}
The finite projective space $\mathrm{PG}(n,q)$ admits a geometric $t$-spread 
if and only if there exists a positive integer $s$ such that $n+1=(t+1)(s+1)$ holds.
\end{theorem}

\begin{lemma} 
\label{regular}
If $\mathrm{PG}(n,q)$ admits a geometric line spread $\mathscr{S}$ then 
$\mathscr{S}_{\left\langle X,Y \right\rangle}$ is a regular line spread of 
the $3$-dimensional subspace $\left\langle X,Y \right\rangle$ of $\mathrm{PG}(n,q)$ 
for any $X,Y\in\mathscr{S}$ ($X\not=Y$). 
\end{lemma}

Combining the cited results of Beutelspacher and Segre, we prove a lemma that plays a
crucial role in the proof of the lower bound in Theorem \ref{thm:uno}.

If $n=3\cdot 2^i-1$ $(i=1,2,\dots )$ then 
$n+1=(2^{i}-1+1)(2+1),$ hence the projective space $\mathrm{PG}(n,q)$ admits a 
geometric $t$-spread $\mathscr{S}^t$ with $t=2^{i}-1$. 
The large points and large lines form a projective plane
$\Pi_{\mathscr{S}^t}$ of order $q^{t+1}$. 
Consider the lines of $\Pi_{\mathscr{S}^t}$ and denote the corresponding
$(2^{i+1}-1)$-subspaces of $\mathrm{PG}(n,q)$ by $\mathfrak{L}_j$ 
($j=1,\dots, q^{2t+2}+q^{t+1}+1$). 
The $t$-spread $\mathscr{S}^t$ is geometric, therefore for all $j$ the elements $X$ of $\mathscr{S}^t$ 
with $X\cap \mathfrak{L}_j\not=\emptyset$ form a $t$-spread of $\mathfrak{L}_j$ 
which will be denoted by $\mathscr{S}^t_{j}$.
The spread $\mathscr{S}^t_{j}$ induces a special line packing of  
$\mathfrak{L}_j.$

\begin{lemma} \label{Packing}
Let $\mathrm{PG}(n,q)$ be the finite projective space of dimension $n=3\cdot 2^i-1$ $(i=1,2,\dots )$. 
Then there exists a geometric $t$-spread $\mathscr{S}^t$ with $t=2^{i}-1$ having the property that 
any finite projective subspace $\mathfrak{L}_j$ admits a line packing $\mathscr{P}_j$ 
such that the set of lines contained in the elements of $\mathscr{S}^t_{j}$ is the 
union of elements of some line spreads of $\mathscr{P}_j$.
\end{lemma}

\begin{proof}
Since $n+1=(1+1)((3\cdot 2^{i-1}-1)+1),$ it follows from Theorem \ref{segre} that $\mathrm{PG}(3\cdot 2^i-1,q)$ admits a 
geometric line spread $\mathscr{S}$. The elements of $\mathscr{S}$ and the $3$-subspaces $\left\langle X,Y|X,Y\in \mathscr{S},X\not=Y\right\rangle$ can be considered, respectively, as points and lines of a $(3\cdot 2^{i-1}-1)$-dimensional space $\Pi_{\mathscr{S}}$ of order $q^2$.  
Denote the $3$-subspaces of $\mathrm{PG}(3\cdot 2^i-1,q)$ corresponding to the lines of 
$\Pi_{\mathscr{S}}$ by $\mathfrak{U}_j$, where $j=1,\dots ,\genfrac[]{0pt}{2}{3\cdot 2^{i-1}}{2}_{q^2}.$ 
Since $\mathscr{S}$ is a geometric spread, as a consequence of Lemma \ref{regular}, we have that the elements $X$ of $\mathscr{S}$ with $X\cap \mathfrak{U}_j\not=\emptyset$ 
form a regular line spread of $\mathfrak{U}_j$ 
which will be denoted by $\mathscr{S}_{\mathfrak{U}_j}$. Moreover, by Theorem \ref{Beutelspacher}, we conclude that the $3$-space 
$\mathfrak{U}_j$ admits a packing $\mathscr{P}_j$ such that $\mathscr{S}_{\mathfrak{U}_j} \in \mathscr{P}_j.$ 
For $k=1,2,\dots ,q^2+q$ let $\mathscr{S}_{j,k}$ be the other spreads of $\mathscr{P}_j,$ hence 
$$\mathscr{P}_j=\{ \mathscr{S}_{\mathfrak{U}_j},\mathscr{S}_{j,1},\dots,\mathscr{S}_{j,q^2+q}\} .$$

We claim that the lines contained in the elements of the set
$$\mathcal{P}=\bigcup _{j=1}^{q^4+q^2+1}\left( \mathscr{P}_j\setminus 
\{ \mathscr{S}_{\mathfrak{U}_j}\} \right) \cup \mathscr{S}$$
is equal to the line set of $\mathrm{PG}(n,q).$ The lines of $\mathscr{S}$ obviously appear in $\mathcal{P}$ exactly once.
If a line $\ell \notin \mathscr{S}$, then $\ell $ lies in a unique subspace of type $\mathfrak{L}_j.$ 
Namely, if the lines $e,f,g,h\in \mathscr{S}$ meet $\ell $ then $\ell \subset \left\langle e,f\right\rangle$ and
$\ell \subset \left\langle g,h\right\rangle ,$ but this means that 
$g\cap \left\langle e,f\right\rangle \neq \emptyset$ and 
$h\cap \left\langle e,f\right\rangle \neq \emptyset$. Since $\mathscr{S}$ is geometric
this implies that $g$ and $h$ are contained in $\left\langle e,f\right\rangle $ and
therefore $\left\langle e,f\right\rangle =\left\langle g,h\right\rangle .$
But $\mathcal{P}$ contains exactly one packing of $\mathfrak{L}_j,$ hence each line of
$\mathrm{PG}(n,q)$ appears in $\mathcal{P}$ exactly once.

Now, we prove the statement of the lemma by induction on $i$.

If $i=1$ then it follows from Theorem \ref{segre} that $\mathrm{PG}(5,q)$ admits a 
geometric line spread $\mathscr{S}$. The elements of $\mathscr{S}$ and the $3$-spaces $\left\langle X,Y|X,Y\in \mathscr{S},X\not=Y\right\rangle$ can be considered as points and lines of a plane $\Pi_{\mathscr{S}}$ of order $q^2$, respectively.
Denote the $3$-spaces of $\mathrm{PG}(5,q)$ corresponding to the lines of 
$\Pi_{\mathscr{S}}$ by $\mathfrak{L}_j$, where $j=1,\dots , q^4+q^2+1.$ 
Since $\mathscr{S}$ is a geometric spread, Lemma \ref{regular} gives that
the elements $X$ of $\mathscr{S}$ with $X\cap \mathfrak{L}_j\not=\emptyset$ 
form a regular line spread of $\mathfrak{L}_j$ 
which will be denoted by $\mathscr{S}_{\mathfrak{L}_j}$.
Because of Theorem \ref{Beutelspacher} the $3$-space 
$\mathfrak{L}_j$ admits a packing $\mathscr{P}_j$ such that $\mathscr{S}_{\mathfrak{L}_j} \in \mathscr{P}_j.$ 
For $i=1,2,\dots ,q^2+q$ let $\mathscr{S}_{j,1}$ be the other spreads of $\mathscr{P}_j,$ hence 
$$\mathscr{P}_j=\{ \mathscr{S}_{\mathfrak{L}_j},\mathscr{S}_{j,1},\dots,\mathscr{S}_{j,q^2+q}\} .$$

Consider now the case $i>1$ and let us assume that the assertion of Lemma \ref{Packing} is proved for all $i'<i$.
Since $n+1=3\cdot 2^{i}=(1+1)(3\cdot 2^{i-1}-1+1)$, by Theorem \ref{segre}, 
$\mathrm{PG}(n,q)$ admits a geometric $1$-spread $\mathscr{S}$. As before, we consider the elements $X$ of $\mathscr{S}$ and the $3$-subspaces 
$\left\langle X,Y\, | \, X,Y\in \mathscr{S},X\not=Y\right\rangle$ as points and lines of a ($3\cdot 2^{i-1}-1$)-space $\Pi_{\mathscr{S}}$ of order $q^{2}$, respectively.
Denote the lines of $\Pi_{\mathscr{S}}$ by $\mathfrak{V}_k$ $(k=1,\dots,M)$ where $M$ 
is the number of lines of $\Pi_{\mathscr{S}}$. 
The spread $\mathscr{S}$ is geometric, therefore the elements $X$ of $\mathscr{S}$ 
with $X\cap \mathfrak{V}_k\not=\emptyset$ form a spread of $\mathfrak{V}_k$ 
which will be denoted by $\mathscr{S}_{k}$. By Lemma \ref{regular}, $\mathscr{S}_{k}$ 
is a regular spread of $\mathfrak{V}_k$. According to Theorem \ref{Beutelspacher}, 
in all $\mathfrak{V}_k$ there exists 
a packing $\mathscr{P}_{\mathfrak{V}_k}$ of $\mathfrak{V}_k$ 
which contains $\mathscr{S}_k$ as one of its spreads. Let this packing be
\[\mathscr{P}_{\mathfrak{V}_k}=\{\mathscr{S}_{k,0},\dots,\mathscr{S}_{k,q^2+q}\}.\]
with $\mathscr{S}_{k,0}=\mathscr{S}_{k}$.

Hence -by induction- $\Pi_{\mathscr{S}}$ admits a basic construction $\mathcal{C}_{i-1}$ 
with the property that any finite projective subspace $\mathfrak{U}_j$ 
admits a packing $\mathscr{P}_j$ such that the lines contained in the elements of 
$\mathscr{S}^t_{j}$ 
are the union of elements of $\mathscr{P}_j$. Let $\mathscr{S}^t_{j}=\{\mathscr{T}_{j,1},\dots,\mathscr{T}_{j,u}\}$ then $\mathscr{P}_j=\mathscr{S}^t_{j}\cup\{\mathscr{T}_{j,u+1},\dots,\mathscr{T}_{j,v}\}$ where $v$ is the number of $1$-spreads in $\mathfrak{U}_j$ of $\Pi_{\mathscr{S}}$. 

Recall that each line of $\Pi_{\mathscr{S}}$ is a $3$-subspace of $\mathrm{PG}(n,q)$. If $\mathscr{T}_{j,l}=\{\mathfrak{u}_{l(1)},\dots,\mathfrak{u}_{l(w)}\colon 1\leq l\leq v\}$ where $w$ is the number of lines in a $1$-spread of $\mathfrak{U}_j$ (as a subspace of $\Pi_{\mathscr{S}}$), then $\mathscr{T}_{j,l,m}=\{\mathscr{S}_{l(1),m},\dots,\mathscr{S}_{l(w),m}\colon 0\leq m\leq q^2+q\}$ is a $1$-spread of $\mathfrak{U}_j$ (as a subspace of $\mathrm{PG}(n,q)$).

We construct the following packing $\mathscr{P}_j$ of $\mathfrak{U}_j$ (as subspaces of $\mathrm{PG}(n,q)$):

\[\mathscr{P}_j=\underset{l=1}{\overset{u}{\sqcup}}\underset{m=0}{\overset{q^2+q}{\sqcup}}\mathscr{T}_{j,l,m}\cup\underset{l=u+1}{\overset{v}{\sqcup}}\underset{m=1}{\overset{q^2+q}{\sqcup}}\mathscr{T}_{j,l,m}.\]
By construction, the set $\underset{l=1}{\overset{u}{\sqcup}}\underset{m=0}{\overset{q^2+q}{\sqcup}}\mathscr{T}_{j,l,m}$ contains all the lines of $\mathscr{S}^t_{j}$ and the lemma follows.
\end{proof}


\section{On line colorings of projective spaces} \label{section3}

First, we introduce some notions that we use to prove our results. 
Let $\mathcal{L}$ be the set of lines of $\mathrm{PG}(n,q)$ and $\mathcal{P}$ be its set of points. 
Given a coloring $\varsigma\colon \mathcal{L} \to [k]$ with $k$ colors, we say that a point $p\in  \mathcal{P}$ 
is an \emph{owner} of a set of colors $C\subseteq [k]$  whenever for every $c\in C$ there is a $q\in \mathcal{P}\setminus \{ p\} $ such that $\varsigma(\left\langle p,q\right\rangle) = c$.
Therefore, $\varsigma$ is a complete coloring if for every pair of colors in $[k]$ 
there is a point in $\mathcal{P}$ which is an owner of both colors.

\subsection{Lower bound}

Now we are ready to prove the lower bound in Theorem \ref{thm:uno}. 
\begin{proof}[Proof of Theorem \ref{thm:uno}, Part 1] Throughout the proof
we use the notations of Section \ref{ProjectiveSpaces}.
Consider the geometric $t$-spread $\mathscr{S}^t$ constructed in Lemma 
\ref{Packing}. Let 
$N=q^{2(t+1)}+q^{t+1}+1$ denote the number of large lines of the corresponding projective plane $\Pi_{\mathscr{S}^t}.$
The space $\mathrm{PG}(n,q)$ admits a basic construction 
$\mathcal{C}_i$ with the property that any finite projective subspace $\mathfrak{L}_j$ admits a packing $\mathscr{P}_j$ 
such that the set of lines contained in the elements of $\mathscr{S}^t_{j}$ is 
the union of elements  of $\mathscr{P}_j$.

Let $r=\genfrac[]{0pt}{2}{t}{1}_q,$ and 
$\mathcal{U}_{j}=\{ \mathscr{S}_{j,1},\dots,\mathscr{S}_{j,r}\} $ 
denote the set of $1$-spreads from $\mathscr{P}_j$ whose union is the
set of all lines that are contained in the elements of $\mathscr{S}^t_{j}.$
Then $\mathscr{P}_j=\mathcal{U}_{j}\sqcup\{\mathscr{S}_{j,r+1},\dots,\mathscr{S}_{j,s}\}$ where $s=\genfrac[]{0pt}{2}{2t+1}{1}_q$ is the number of $1$-spreads of $\mathscr{P}_j$. Note that the number of $1$-spreads in $\mathscr{P}^*_j:=\mathscr{P}_j\setminus\mathcal{U}_{j}$
is $s-r=q^t\genfrac[]{0pt}{2}{t+1}{1}_q.$
Every element $X$ of $\mathscr{S}^t$ is a $t$-subspace, hence, by Theorem \ref{Beutelspacher2}, admits a packing $\mathscr{P}_X=\{\mathscr{S}_{1,X},\dots,\mathscr{S}_{r,X}\}$.
Then the set of lines contained in the elements of the set
$$\left( \bigcup _{j=1}^{N}\mathscr{P}_j^*\right) \cup  
\left( \bigcup _{X\in \mathscr{S}^t}\mathscr{P}_X\right) $$
is equal to $\mathcal{L}.$

Now we define the coloring. We distinguish the two types of spreads.
For a fixed $1\leq j\leq N$ and $r+1\leq k \leq s$ let the lines of 
$\mathscr{S}_{j,k}$ be colored with the color $c_{j,k}=(k-r-1)N+j$. 
This implies that each point of $\mathfrak{L}_j$ is owner of the colors 
$c_{j,k}$ for all $k.$
For a fixed $1\leq m\leq r$ let the lines of 
$\mathscr{S}_{m,X}$ be colored with the color $c_m=(s-r)N+m$ for all $X.$
As $\mathscr{S}_{m,X}$ is a 1-spread of $X\in \mathscr{S}^t,$ the set of 
points on the lines of the set  $\cup _{X\in \mathscr{S}^t}\mathscr{S}_{m,X}$
is equal to $\mathcal{P}.$ Thus each point of $\mathrm{PG}(n,q)$ is owner 
of the colors $c_{m}$ for all $m.$

Observe that the coloring is proper by definition.
We claim that it is also complete. We have to show that
for every pair of colors $\{c, c'\} $ 
there is a point $x$ of $\mathrm{PG}(n,q)$, which is an owner of  both $c$ and $c'$.
This is obvious when at least one of $c$ and $c'$ is of type $c_m.$
Suppose that $c=c_{j,k}$ and $c'=c_{j',k'}.$ Take the subspace  
$\mathfrak{L}_j\cap \mathfrak{L}_{j'}.$ Its dimension is $2t+1$ or $t,$ 
according as $j=j'$ or $j\neq j',$ so it is not the empty set. 
Any point $x\in \mathfrak{L}_j\cap \mathfrak{L}_{j'}$ is the owner of both colors.

In the coloring we use 
$$(s-r)N+r=q^t\frac{q^{n+1}-1}{q-1}+\frac{q^t-1}{q-1}=\frac{q^{n+t+1}-1}{q-1}$$ 
colors. Let $h=\frac{4n+1}{3n}.$ Since 
$n+t+1=\frac{4n+1}{3}=hn$ and $2q^n>\frac{q^{n+1}-1}{q-1}=v,$ we have 
$$\frac{q^{n+t+1}-1}{q-1}=\frac{q^{hn}-1}{q-1}> \frac{1}{2^h}\frac{(2q^n)^{h}}{q}>\frac{1}{2^h}\frac{v^h}{q},$$
hence Inequality (\ref{main1}) of Theorem \ref{thm:uno} holds with $c_n=\frac{1}{2^h},$
and the theorem follows, because $5\leq n$ implies $\frac{4}{3}< h\leq \frac{7}{5}.$ 
\end{proof}

\subsection{Upper bound}\label{sec:upper_bound}

Now we prove the upper bound for the pseudoachromatic index of $\mathrm{PG}(n,q)$.

\begin{proof}[Proof of Theorem \ref{thm:uno}, Part 2]
If $r$ denotes the number of lines through a fixed point, then the total number of 
unordered line-line incidences 
is $v\tbinom{r}{2}$. Hence $v\tbinom{r}{2}\geq \tbinom{\psi'(\mathrm{PG}(n,q))}{2}.$ 
Solving this quadratic inequality we get 
\[\psi'(\mathrm{PG}(n,q))\leq \frac{1+\sqrt{1+4vr(r-1)}}{2}.\]
Since $\sqrt{1+4vr(r-1)}< \sqrt{4vr^2}-1$ and $r=\tfrac{v-1}{q}$, 
this gives 
\[\psi'(\mathrm{PG}(n,q)) <\sqrt{v}r=\frac{1}{q}\sqrt{v}(v-1)\] 
and the result follows.
\end{proof}


\section{The case of $\mathrm{PG}(3,2)$}\label{new}

In this section, we determine the pseudoachromatic index of the smallest finite projective space, 
$\mathrm{PG}(3,2)$, in a pure combinatorial way, without using any computer aided calculations. 
To do this, we need some lemmas about pencils and null polarities.

\begin{definition}
Let $\Pi $ be a plane and $P\in \Pi $ be a point in $\mathrm{PG}(3,q).$
A \emph{pencil} with carrier $P$ in $\Pi $ is the set of the $q+1$ lines of $\mathrm{PG}(3,q)$
through $P$ that are contained in $\Pi .$
\end{definition}   

\begin{lemma}
\label{otbolpencil}
Let $\mathcal{E}$ be a set of five lines in $\mathrm{PG}(3,2).$ If any two lines of $\mathcal{E}$ 
have a point in common then $\mathcal{E}$ contains a pencil.
\end{lemma}   

\begin{proof}
Any two lines of $\mathcal{E}$ meet, hence, all the lines in $\mathcal{E}$ are either coplanar or all of them have a point in common.
It follows from Proposition \ref{combprop} that in $\mathrm{PG}(3,2)$ there are seven lines through each
point, and dually, each plane contains seven lines.
Because of the duality, we may assume without loss of generality that the five lines of $\mathcal{E}$ are
coplanar. As each plane contains seven points and $\binom{5}{2}>7,$ at least three lines of 
$\mathcal{E}$ have a point in common, thus 
$\mathcal{E}$ contains a pencil.
\end{proof}

\begin{definition}
Let $\mathrm{PG}(3,q)'$ denote the dual space of $\mathrm{PG}(3,q),$ and let 
$A$ be a $4\times 4$ non-singular matrix over 
$\mathrm{GF}(q)$ that satisfies the equation $A=-A^{\mathrm{T}}$ and
whose all diagonal elements are $0.$

A \emph{null polarity} $\pi \colon \mathrm{PG}(3,q)\rightarrow \mathrm{PG}(3,q)'$ is a collineation which maps 
the point with coordinate vector $\mathbf{x}$ to the point with coordinate vector $\mathbf{x}A.$ 
\end{definition}   

As the points, lines and planes of the dual space are planes, lines and points of the original space, respectively, 
a null polarity maps lines of $\mathrm{PG}(3,q)$ 
to lines of $\mathrm{PG}(3,q)$. 
A null polarity maps intersecting lines to intersecting lines, hence 
the proof of the following statement is straightforward.

\begin{lemma}
\label{haromszogbe}
Let $\pi$ be a null polarity and $\varsigma$ be a 
line-coloring of $\mathrm{PG}(3,q)$. Then $\varsigma$ is complete 
if and only if $\varsigma\circ\pi^{-1}$ is a complete line-coloring 
of $\mathrm{PG}(3,q)'$. 
\end{lemma}

\begin{theorem} 
The pseudoachromatic index of $\mathrm{PG}(3,2)$ is equal to $18$, i.e., 
\[\psi ' (\mathrm{PG}(3,2))=18.\]
\end{theorem}

\begin{proof}
In $\mathrm{PG}(3,2)$ there are three points on each line and there are seven lines through each point,
hence the total number of lines intersecting a fixed line is $3\cdot (7-1)=18.$ Thus if a complete line
coloring contains a color class of size one then the coloring cannot contain more than $1+18=19$ color classes. 
There are 35 lines in $\mathrm{PG}(3,2),$ so the number of color classes containing at least two lines is at
most $\lfloor 35/2\rfloor =17.$ Hence $\psi ' (\mathrm{PG}(3,2))\leq 19.$ 

Now, we prove that $\psi ' (\mathrm{PG}(3,2))\leq 18$.
Suppose to the contrary that there exists a complete coloring $\varsigma$ of $\mathrm{PG}(3,2)$ with $19$
color classes. 

We claim that $\varsigma$ contains three or four color classes of size one 
and no three of the corresponding lines form a pencil.   
By the pigeonhole principle, there are at least 3 color classes of size one. If 
there were at least five color classes of size one in $\varsigma$ then, by Lemma \ref{otbolpencil},
we could choose three color classes such that the corresponding 
lines would form a pencil. 
Suppose that three lines, $\ell _1 ,\ell_2$ and $\ell_3$ 
form a pencil with carrier $P$ in the plane $\Pi ,$
and each of these lines forms a color class of size one. 
Consider the other 16 classes. 
At most 4 of them contain lines through $P$ and at most 
4 of them contain lines in $\Pi .$ Each of the remaining 
at least 8 classes must have size at least 3, 
because they have to meet each $\ell _i$ for $i=1,2,3.$
This implies that these color classes contain altogether $8\times 3=24$
or more lines. As the total number of lines is $35,$ 
this means that each of the remaining $11$ 
color classes contains exactly one line. Hence each of the seven lines 
through $P,$ and each of the $4$ lines in $\Pi $ not through $P$ 
are color classes of size one, but they do not meet, so $\varsigma$ is not 
complete. This contradiction proves the statement.

Choose three color classes of size one and let $\ell _1 ,\ell_2$ 
and $\ell_3$ be the lines in these color classes. Any two of these lines 
have a point in common, but they do not form a pencil, hence 
either they form a triangle, or they have a point in common but they are 
not coplanar. In the latter case apply Lemma \ref{haromszogbe}. 
If the three lines meet in the point $P$ then
after a null polarity $\pi ,$ the lines $\ell ^{\pi }_1 ,\ell ^{\pi }_2$ 
and $\ell ^{\pi }_3$ form a triangle in the plane $P^{\pi }$. 
As $\mathrm{PG}(3,2)$ is isomorphic to its dual space,
it is enough to consider the first case.

From now on, we suppose that $\ell _1 ,\ell_2$ and $\ell_3$ form a triangle $ABC$ 
in the plane $\Pi .$  Let $A',B'$ and $C'$ be the 
third points of the sides of the triangle, respectively, and let 
$D=AA'\cap BB'\cap CC'$ be the seventh point of the plane $\Pi .$ 
Take $\Pi $ as the plane at infinity and consider the remaining
eight points as $\mathrm{AG}(3,2).$ The coordinates of the points in $\Pi $ can be choosen as follow.

\noindent
$A=(0:1:0:0),$ $B=(0:0:1:0),$ $C=(0:0:0:1),$  
$A'=(0:0:1:1),$ 

\noindent
$B'=(0:1:0:1),$ $C'=(0:1:1:0)$ and $D=(0:1:1:1).$

First, suppose that there is a $4^{\mathrm{th}}$ color class of size one and let $\ell _4$ denote the line in this class.  
Then $\ell _4$ must be in $\Pi .$ If it contains one of the points
$A,B$ or $C,$ then a pencil appears, hence the coloring is not complete. So
we may assume that $\ell _4$ is the line $A'B'C'.$ Among the other 15 color classes there are 14 classes of size 2 and one of size 3.
Consider the four lines, say $\ell _5 ,\ell_6, \ell _7$ and $\ell_8,$ 
through $D$ but not in $\Pi .$ If two or three of them formed a color class, then this class would have empty intersection with each
of $\ell _1 ,\ell_2$ and $\ell_3,$ contradiction. So these four lines 
are distributed among at least three color classes and each class of size two 
must contain a line of $\Pi .$ Thus there are two possibilities for these 
color classes:
\begin{itemize}
\item[(a)]
$\{\ell _5 ,\ell_8,AA'\} ,$ $\{\ell _6 ,BB'\} ,$ $\{\ell _7 ,CC'\} ,$
\item[(b)]
$\{\ell _5 ,AA'\} ,$ $\{\ell _6 ,BB'\} ,$ $\{\ell _7 ,CC'\} ,$ 
$\{\ell _8, \ell _9, \ell _{10}\} ,$ where $\ell _9$ is a line through $A$ and
$ \ell _{10}$ is a line through $A'.$
\end{itemize}
Each of the remaining classes contains two lines whose points at infinity cover
$\ell _1 ,\ell_2,\ell_3$ and $\ell_4$. Since no three of these lines have a point in common, each of the remaining classes is incident to $\ell _1 ,\ell_2,\ell_3$ and $\ell_4$ if and only if two points at infinity of these color classes 
coincide with one of the sets $\{ A,A'\} , \{ B,B'\} $ and $\{ C,C'\} .$

If there is no more color class of size one,
then each of the remaining 16 classes has size 2. The pairs of the four lines through $D$
must be the four lines of $\Pi $ distinct from $\ell _1 ,\ell_2,\ell_3.$ If an affine line passes  
on $A',$ then its pair must pass on $A,$ and the same is true for the lines through $\{ B,B'\} $
and $\{ C,C'\} .$

We can summarize these possibilities as follow.

\begin{itemize}
\item
Each of the lines $\ell _1 ,\ell_2$ and $\ell_3$ forms a class of size one.
\item
There are 12 classes such that two points at infinity of these color classes 
coincide with one of the sets $\{ A,A'\} , \{ B,B'\} , \{ C,C'\} .$ 
\item
Each of the pairs $\{\ell _5 ,AA'\} ,$ $\{\ell _6 ,BB'\} $ and $\{\ell _7 ,CC'\} $ 
belong to one color class.
\item
The nineteenth color class contains the line $A'B'C'.$ 
\item
The line $\ell _8$ is ``free".   
\end{itemize}

\smallskip

We can choose the system of reference such that the pair of $AA'$
is the line $DOE$ where $O=(1:0:0:0)$ and $E=(1:1:1:1)$. Let $X=(1:1:0:0)$, $Y=(1:0:1:0),$ 
$Z=(1:0:0:1),$ $K=(1:1:1:0),$ $L=(1:1:0:1)$ and $M=(1:0:1:1)$
be the other affine points of $\mathrm{PG}(3,2)$, see Figure \ref{Fig1}.  The pair 
of the line $CXL$ is either the line $C'O$ or $C'E$. 
As the roles of $O$ and $E$ were symmetric previously, we may assume  
without loss of generality that $C'OK$ is the pair of $CXL$.
\begin{figure}[!htbp]
\begin{center}
\includegraphics[scale=0.73]{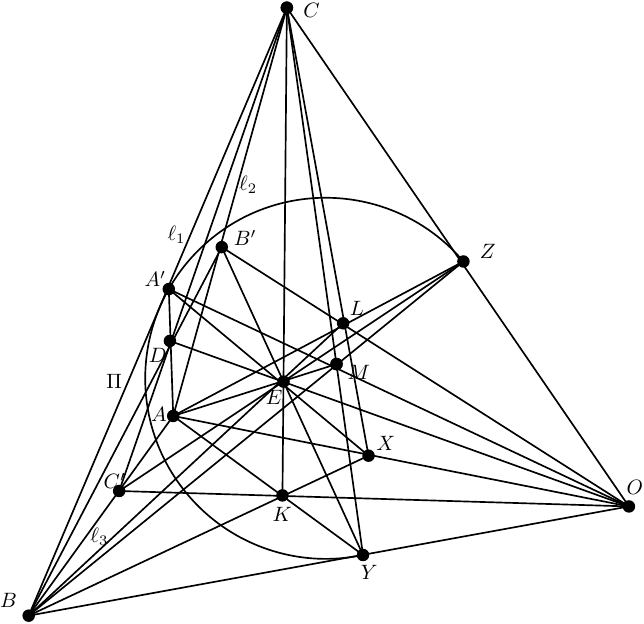}
\caption{\label{Fig1} $\mathrm{PG}(3,2)$, not all lines shown.}
\end{center}
\end{figure}

First, consider the three other classes whose two points in $\Pi $ are $C$ and $C'.$
The affine part of the three lines through
$C$ are $OZ,$ $MY,$ $KE,$ while the affine part of the three lines through $C'$
are $EZ,$ $XY,$ $LM.$
Each of these classes must meet the line $DOE.$ Hence, we need a matching between these two line-triples
such that each pair contains at least one of the points $O$ and $E.$ So the pair of $MY$
must be $EZ.$ There are two possibilities for the remaining two pairs, so the four possible pairs
through $C$ and $C'$ are:
\begin{itemize}
\item[i)]
$(XL, \,OK), \, (MY, \, EZ), \, (OZ, \, XY), \, (KE, \, LM),$
\item[ii)] 
$(XL, \,OK), \, (MY, \, EZ), \, (KE, \, XY), \, (OZ, \, LM).$
\end{itemize}

In Case i) take the four classes whose two points in $\Pi $ are $B$ and $B'.$
The affine parts of the four lines through
$B$ are $OY,$ $EL,$ $XK,$ $ZM,$ while the affine parts of the four lines through $B'$
are $OL,$ $EY,$ $XZ,$ $MK.$ Again, we need a matching such that
each pair contains at least one of the points $O$ and $E,$ and each class 
must meet the four classes belonging to $\{ C,C'\} .$ So the pair of $XK$ is $EY,$
because $(XK, \, OL)$ has empty intersection with $(MY, \, EZ).$
Hence the pair of $ZM$ is $OL.$ The pair of $MK$ is $OY,$
because $(EL, \, MK)$ has empty intersection with $(OZ, \, XY).$
So the affine parts of the four pairs belonging to $\{ B,B'\} $ are
$$(XK, \,EY), \, (ZM, \, OL), \, (OY, \, MK), \, (EL, \, XZ).$$

Take the four classes whose two points in $\Pi $ are $A$ and $A'.$
The affine parts of the four lines through
$A$ are $OX,$ $EM,$ $YK,$ $ZL,$ while the affine part of the four lines through $A'$
are $OM,$ $EX,$ $YZ,$ $LK.$
At least three classes consist of only two lines. Let us look for these classes.
None of the pairs 
$(OX, \, LK),$ $(OX, \, YZ),$ $(EM, \, YZ),$ $(EM, \, LK)$ is good, because
its intersection is empty with 
$(YM, \, EZ),$ $(KE, \, LM),$ $(XL, \, OK),$ $(OZ, \, XY),$ respectively.
In the same way none of the pairs 
$(KY, \, OM),$ $(KY, \, EX),$ $(LZ, \, OM),$ $(LZ, \, EX)$ is good because
their intersections are empty with 
$(EL, \, XZ),$ $(ZM, \, LO),$ $(XK, \, YE),$ $(OY, \, MK),$ respectively.
This means that in the matching there are only four possible pairs containing $OX$ or $EM,$ namely
\begin{equation}
\label{masodik}
(OX, \, OM), \, (OX, \, EX), \, (EM, \, OM), \,(EM, \, EX),
\end{equation}  
and four possible pairs containing $KY$ or $LZ,$ namely
\begin{equation}
\label{elso}
(KY, \, LK), \, (KY, \, YZ), \, (LZ, \, YZ), \,(LZ, \, LK).
\end{equation}
Thus at least one pair from (\ref{masodik}) and  at least one pair from (\ref{elso})
form a color class. 

Now consider the affine part of the two color classes containing $\{ B,B'\} $ and $\{ C,C'\} .$ These are
the lines through $D,$ except $DOE.$ So they consist of the points $X$ and $M,$
$Z$ and $K,$ $L$ and $Y.$ At least one of the two classes contains only one pair of
points. But the pair $\{ X,M\} $ has empty intersection with any class from (\ref{elso}), 
while both pairs $\{ Z,K\} $ and $\{ L,Y\} $ have empty intersection with any class from (\ref{masodik}).
Hence the coloring cannot be complete in Case i).

\smallskip

Now consider Case ii).
Take the four classes whose two points in $\Pi $ are $B,B'.$
The affine parts of the four lines through
$B$ are $OY,$ $EL,$ $XK,$ $ZM,$ while the affine parts of the four lines through $B'$
are $OL,$ $EY,$ $XZ,$ $MK.$ Again we need a matching such that
each pair contains at least one of the points $O$ and $E,$ and each class 
must meet the four classes belonging to $\{ C,C'\} .$ 

So the pair of $XK$ is $OL,$
because $(XK, \, EY)$ has empty intersection with $(OZ, \, LM).$
Hence the pair of $ZM$ is $EY.$ 
We distinguish two cases, according to the pair of $OY.$
So the affine parts of the four pairs belonging to $\{ B,B'\} $ are
\begin{itemize}
\item[(a)]
$(XK, \,LO), \, (ZM, \, EY), \, (OY, \, XZ), \, (EL, \, MK),$
\item[(b)]
$(XK, \,LO), \, (ZM, \, EY), \, (OY, \, MK), \, (EL, \, XZ).$
\end{itemize}
Take the four classes whose two points at infinity are $A$ and $A'.$
The affine parts of the four lines through
$A$ are $OX,$ $EM,$ $YK,$ $ZL,$ while the affine part of the four 
lines through $A'$ are $OM,$ $EX,$ $YZ,$ $LK.$ At least three classes 
consist of only two lines. Let us look for these classes.
In both cases none of the pairs 
$(OX, \, LK),$ $(KY, \, EX),$ $(EM, \, YZ),$ $(LZ, \,OM )$ is good, because
it has empty intersection with 
$(YM, \, EZ),$ $(OZ, \, LM),$ $(XK, \, LO),$ $(KZ, \, XY),$ respectively.

Furthermore, in Case (a) none of the pairs 
$(OX, \, YZ),$ $(EM, \, LK)$ is good, because
its intersection is empty with 
$\{ O,X, Y,Z\} \cap \{E,L,M,K\} =\{E,M,L,K\} \cap \{O,Y,X,Z\} =\emptyset .$
Thus we get four possible pairs containing $OX$ or $EM:$
\begin{equation}
\label{harmadik}
(OX, \, OM), \, (OX, \, EX), \, (EM, \, OM), \,(EM, \, EX),
\end{equation}
and six possible pairs containing $KY$ or $LZ:$
$$(KY, \, OM), \, (LZ, \, EX),\,$$
\begin{equation}
\label{negyedik}
 (KY, \, LK), \, (KY, \, YZ), \, (LZ, \, LK), \,(LZ, \, YZ).
\end{equation}
If either $(KY, \, OM)$ or $(LZ, \, EX)$ belongs to the matching, then only one more pair from 
(\ref{harmadik}) can be in it, hence at least one more pair from (\ref{negyedik}) also
belongs to the matching. 
Thus at least one pair from (\ref{harmadik}) and at least one pair from (\ref{negyedik})
form a color class.

In Case (b) none of the pairs 
$(LZ, \, EX),$ $(KY, \, OM)$ is good, because
it has empty intersection with 
$(OY, \, MK),$ $(EL, \, XZ),$ respectively.
Thus we get six possible pairs containing $OX$ or $EM:$
$$(OX,\, YZ),\, (EM, \, LK),$$ 
\begin{equation}
\label{otodik}
(OX, \, OM), \, (OX, \, EX), \, (EM, \, OM), \,(EM, \, EX),
\end{equation}
and four possible pairs containing $KY$ or $LZ:$
\begin{equation}
\label{hatodik}
 (KY, \, LK), \, (KY, \, YZ), \, (LZ, \, LK), \,(LZ, \, YZ).
\end{equation}
If either $(OX, \, YZ)$ or $(EM, \, LX)$ belongs to the matching, then only one more pair from 
(\ref{otodik}) can be in it, hence at least one more pair from (\ref{hatodik}) also
belongs to the matching. 
Thus at least one pair from (\ref{otodik}) and  at least one pair from (\ref{hatodik})
form a color class. 

Finally, in both Cases (a) and (b), consider the affine part of the 
two color classes containing $\{B,B'\} $ and $\{C,C'\} .$ 
These are
the lines through $D,$ except $DOE.$ So they consist of the points $X$ and $M,$
$Z$ and $K,$ $L$ and $Y.$ At least one of the two classes contains only one pair of
points. But the pair $\{ X,M\} $ has empty intersection with any class from (\ref{negyedik}) and from (\ref{hatodik}), 
while both of the pairs $\{Z,K\} $ and $\{ L,Y\} $ have empty intersection with any class 
from (\ref{harmadik}) and from (\ref{otodik}).
Hence the coloring cannot be complete in Case ii).

\smallskip

Now, we present a complete coloring with 18 color classes. 

Let the lines $\ell _1 ,\ell_2,\ell_3$ and $\ell _4=A'B'C'$ form color classes of size one. 
These classes are denoted by $\mathcal{C}_1,\mathcal{C}_2,\mathcal{C}_3$ 
and $\mathcal{C}_4,$ respectively. The class $\mathcal{C}_5$ consists of the lines $AA'D$ and $OED,$  while the class $\mathcal{C}_6$ consists of 
the remaining five lines through $D.$ Any two of these six classes obviously have non-empty intersection.
The remaining twelve classes of size two are formed by the $3\times 4$ pairs of lines whose points at infinity are $\{ A,A'\},$ 
$\{ B,B'\} $ and
$\{ C,C'\} ,$ respectively. The affine parts of these classes are the following:
$$\mathcal{C}_{A1} \colon \, (OX,EX), \quad  \mathcal{C}_{A2} \colon \, (OM,EM),\quad   
\mathcal{C}_{A3} \colon \, (YK,YZ),\quad   \mathcal{C}_{A4} \colon \, (LZ,LK);$$
$$\mathcal{C}_{B1} \colon \, (OY,MX), \quad  \mathcal{C}_{B2} \colon \, (XK,EY),\quad   
\mathcal{C}_{B3} \colon \, (ZM,OL),\quad   \mathcal{C}_{B4} \colon \, (EL,XZ);$$
$$\mathcal{C}_{C1} \colon \, (OZ,XY), \quad  \mathcal{C}_{C2} \colon \, (XL,OK),\quad   
\mathcal{C}_{C3} \colon \, (YM,EZ),\quad   \mathcal{C}_{C4} \colon \, (KE,LM).$$
If $1\leq i\leq 6$ then $\mathcal{C}_i$ contains at least one element of each of the 
pairs $\{ A,A'\} ,$ $\{ B,B'\} $ and $\{ C,C'\} ,$ and any two color classes belonging to the same
quadruple of classes of type $\mathcal{C}_{Qi}$ also intersect each other. Hence it is enough to show that   
$\mathcal{C}_{Qi}$ and $\mathcal{C}_{Rj}$ have non-empty intersection if $Q\neq R.$ The three parts of Table \ref{cases} give one point of intersection in each case.
\begin{table}[h]
\label{cases}
\begin{center}
\begin{tabular}{l|cccc|}
  & $\mathcal{C}_{B1}$ & $\mathcal{C}_{B2}$ & $\mathcal{C}_{B3}$ & $\mathcal{C}_{B4}$  \\
\hline
$\mathcal{C}_{A1}$ & O & X & O &  X \\
$\mathcal{C}_{A2}$ & M & E & M &  E \\
$\mathcal{C}_{A3}$ & Y & Y & Z & Z  \\
$\mathcal{C}_{A4}$ & K & K & L & l \\
\hline
\end{tabular}
\quad \quad
\begin{tabular}{l|cccc|}
  & $\mathcal{C}_{C1}$ & $\mathcal{C}_{C2}$ & $\mathcal{C}_{C3}$ & $\mathcal{C}_{C4}$  \\
\hline
$\mathcal{C}_{A1}$ & O & O & E &  E \\
$\mathcal{C}_{A2}$ & O & O & M &  M \\
$\mathcal{C}_{A3}$ & Y & K & Y & K  \\
$\mathcal{C}_{A4}$ & Z & L & Z & l \\
\hline
\end{tabular}

\bigskip
\medskip
\begin{tabular}{l|cccc|}
  & $\mathcal{C}_{C1}$ & $\mathcal{C}_{C2}$ & $\mathcal{C}_{C3}$ & $\mathcal{C}_{C4}$  \\
\hline
$\mathcal{C}_{B1}$ & O & O & M &  M \\
$\mathcal{C}_{B2}$ & X & X & E &  E \\
$\mathcal{C}_{B3}$ & O & O & M & M  \\
$\mathcal{C}_{B4}$ & X & X & E & E \\
\hline
\end{tabular}
\caption{\label{cases} Points of intersections of $\mathcal{C}_{Qi}$ and $\mathcal{C}_{Rj}$.}
\end{center}
\end{table}

This proves that the coloring is complete.
\end{proof}

\end{document}